\documentclass[11pt]{article}

\usepackage{amsmath, amscd, amsfonts, amssymb, amsthm, tikz}
\newcommand{\ind}{\mathbf{1}}

\newtheorem{theorem}{Theorem}[section]
\newtheorem{corollary}[theorem]{Corollary}

\newtheorem{lemma}[theorem]{Lemma}
{\bf}{\rm}
\newtheorem{definition}{Definition}
\numberwithin{theorem}{section}

\def \R{\mathbb{R}}
\def \Z{\mathbb{Z}}
\def \N{\mathbb{N}}

\def \aa{\alpha}
\def \bb{\beta}

\def \eps{\epsilon}

\def \th{\theta}

\def \CC{{\cal C}}
\def \DD{{\cal D}}

\def \E{\mathbf{E}}

\def \ff{\infty}

\def \({\left(}
\def \){\right)}

\def \lc{\left\{}
\def \rc{\right\}}

\def \nn{\nonumber}

\def \bs{\begin{slide} }
\def \es{\end{slide} }

\def \sas{S$\alpha$S }

\def \fh{{f}_h}

\def\B{v}
\def \L^\alpha{L^\alpha}

\def \beq{\begin{equation}}
\def \ee{\end{equation}}
\def \bea{\begin{eqnarray}}
\def \eea{\end{eqnarray}}
\def \bes{\begin{eqnarray*}}
\def \ees{\end{eqnarray*}}

\newcommand{\cE}{{\mathcal E}}

\newcommand{\cS}{{\mathcal S}}

\newcommand{\bbE}{\mathbb{E}}

\newcommand{\bbP}{\mathbb{P}}

\newcommand{\bbR}{\mathbb{R}}

\newcommand{\bbZ}{\mathbb{Z}}

\begin{document}

\title{Approximation of stable random measures and applications to linear fractional stable integrals.}
\author{ Cl\'ement Dombry\footnote{Laboratoire LMA, CNRS UMR 7348, Universit\'e de Poitiers,
 T\'el\'eport 2, BP 30179, F-86962 Futuroscope-Chasseneuil cedex, France.
 Email: clement.dombry@math.univ-poitiers.fr} \protect\hspace{1cm}
\quad Paul Jung\footnote{Department of Mathematics, University of Alabama
 Birmingham, USA.
 Email: pjung@uab.edu} \protect\hspace{1cm}}

\maketitle

\abstract{

 Using lattice approximations of $\R^d$, we develop a way to approximate stable processes that are represented by stochastic integrals over $\R^d$. Via a stable version of the Lindeberg-Feller Theorem we show that the
approximations weakly converge as the mesh-size goes to zero. As an application, we improve upon previous approximation schemes for integrals with respect to
 linear fractional stable motions.}

\vspace{0.5cm}



\tableofcontents

\section{Introduction}
Stable integration is an important tool in the theory of $\aa$-stable processes. Similar to the theory for Gaussian processes, it is known (\cite[Sec. 13.2]{samorodnitsky1994stable}) that all stable processes, satisfying mild conditions, can be constructed from integrals of the form
\beq\label{stableprocess}
X_t=\int_E f_t(x) M_\alpha(dx),\quad t\in T,
\ee
where $M_\alpha$ is  an independently scattered $\alpha$-stable random measure on the measurable space $(E,\cE)$ with control measure $m$ and $(f_t)_{t\in T}$ is a kernel such that $f_t\in L^\alpha(E,\cE,m)$ for all $t\in T$. If $T=\R$ and $f_t=1_{[0,t]}$ then $X_t$ is an $\aa$-stable Levy motion having independent and stationary increments (the symmetric case is the stable analog of Brownian motion).

In this work, we approximate the finite-dimensional distributions of \eqref{stableprocess} using a Riemann sum-type scheme.
These approximations are useful for the dual purposes of  intuition  and simulation of stable processes.
 The weak convergence of our scheme is facilitated by a Lindeberg-Feller type stable limit theorem, which we have not previously seen in the literature. 
 
 A couple of different discrete approximations of stable processes have appeared previously in the literature.  One approach is  Lepage's series  which was improved upon in a 
 series of papers by J. Rosinski (see \cite{rosinski2001series} and the references therein). In the present paper, we use a lattice approximation of stable integrals which extends, 
 to $f\in L^\aa(\R^d)$, the ``moving-average" discrete approximations of L-FSMs in \cite{davydov,maejima1983class, astrauskas1983limit, davis1985limit} corresponding 
 to the case $f_t=1_{[0,t]}$.  The work of \cite{kasahara1988weighted} improved upon these earlier papers to obtain  discrete approximations of slightly more general stable 
 processes, while \cite{avram1992weak} showed that tightness of discretized L-FSMs cannot be achieved in the $J_1$-Skorokhod topology. In \cite{kokoszka1995fractional}, 
 it was shown that discretized L-FSMs satisfy the fractional ARIMA equations
 and a closer look at issues concerning absolute convergence was taken. 

A secondary purpose of this work is to generalize certain Gaussian integrals  to the $\aa$-stable case and, as in \cite{kasahara1988weighted}, we then approximate such 
integrals with the scheme just described.
In the past fifteen years or so, there has been an effort to develop stochastic integrals with respect to a broader class of Gaussian processes than just Brownian motion. 
In particular, consider Gaussian processes with stationary increments, but replace the independent increments condition with the weaker condition of self-similarity.  
Normalizing the variance at $t=1$ to
unity, one gets the single parameter family of fractional Brownian motions (FBM) with Hurst  self-similarity parameter $0<H<1$.

The theory of integration with respect to FBM is difficult because FBM is not a semi-martingale.  Nevertheless, rapid progress has been 
made using several different approaches (with
significant overlap between them).  Roughly speaking, they can be categorized into four approaches which use, respectively, fractional derivatives and integrals, 
Malliavin calculus, fractional white noise theory, and path-wise integration (see \cite{biagini2008}).

In Section \ref{sec:lfsm}, we consider a generalization of the FBM integral based on fractional integro-differentiation to $\aa$-stable 
analogs of FBM called the linear fractional stable motions\footnote{The term {\it linear fractional stable motion} was introduced in \cite{cambanis1989two} 
due to its close relation to linear time series (moving average processes).} (L-FSMs).
By ``$\aa$-stable analog", we mean that a L-FSM is a self-similar, symmetric stable process with stationary increments. 
Any process with these properties is called a FSM.  In contrast to the Gaussian picture, for each admissible $(\aa,H)$ pair, there
is not a unique (normalized) FSM, up to finite-dimensional distributions.  Moreover, for each $(\aa,H)$ pair with $0<H<1$ and $H\neq 1/\aa$, there
are infinitely many L-FSMs. These L-FSMs are represented by \eqref{stableprocess} where $E=\R$ is equipped with Lebesgue measure, $M_\aa$ is symmetric, and
 \bea\label{LFSMkernel}
&&f^{a,b}_t(x):= \\
&&a\((t-x)_+^{H-1/\aa} - (-x)_+^{H-1/\aa}\) +
 b\((t-x)_-^{H-1/\aa} - (-x)_-^{H-1/\aa}\)\nn
 \eea
for properly normalized order pairs $(a,b)$ where $a,b\ge 0$ (see \cite[Sec. 7.4]{samorodnitsky1994stable} for more details). 
Here $x_-=|x|$ if $x<0$ and $0$ otherwise (similarly for $x_+$).
 The family of L-FSMs were the first FSMs to be constructed and studied, and much is known about them.  Our motivation comes partly from
\cite{pipiras2000integration} which handles the $\aa=2$ case. As in their work, we restrict ourselves to
deterministic integrands, but \cite{pipiras2000integration} shows that even in the $\aa=2$ case, the theory for deterministic
integrands is not completely trivial.

An integral with respect to L-FSM will be defined as an integral with respect to a linear fractional stable random measure
which we define for $\aa>1$ and all
permissable Hurst parameters $0<H<1$.
We have recently learned that when $H>1/\aa$, \cite{maejima2008limit} has developed similar integrals and  also discrete approximations for them.
However,  the  convergence results for their approximations concern a strictly smaller class of integrands. In particular, they
require bounded integrands which are piece-wise continuous (we require no continuity or boundedness) and which must satisfy  a faster tail decay than ours.

The rest of the paper is organized as follows. In
Section 2, we review the notion of stable random measures and
present our result concerning the convergence of discretizations of
stable random measures. In Section 3, integrals with respect to
L-FSMs are defined, and their approximation by moving averages of
i.i.d. random variables  are discussed. Section 4 is devoted to the
proofs. 


\section{Discrete approximations of \sas random measures }
A useful viewpoint is that  a random measure is a stochastic process:
\begin{definition}[Random measure]
Let $(E,\cE)$ be a measurable space and $V$ be a vector space of measurable functions $f:E\to\bbR$. A random measure on $(E,\cE)$ 
is a stochastic process $(M[f])_{f\in V}$ satisfying the  linearity property: for all $a_1,a_2\in\bbR$ and $f_1,f_2\in V$,
\begin{equation}\label{eq:lin}
M[a_1f_1+a_2f_2]=a_1M[f_1]+a_2M[f_2] \quad {\rm almost\ surely.}
\end{equation}
\end{definition}
Let us make a few comments concerning this definition. First of all, the linearity property \eqref{eq:lin} ensures that the finite-dimensional distributions 
of the process $(M[f])_{f\in V}$
are determined by its one-dimensional distributions. If  $\ind_A\in V$ for $A\in\cE$, we note $M(A)=M[\ind_A]$
 which is thought of as the random measure of the set $A$.
If $M(A_i)$ are independent for disjoint sets $A_1,\ldots, A_k$, then $M$ is said to be {\it independently scattered}.
For general $f\in V$, to emphasize the analogy with usual integration, the notation $M[f]=\int_E f(x)M(dx)$ is often used.
Finally, if one so pleases, one may also view  the random measure $M$ as a random linear functional on the linear space $V$ (see for example \cite{dudley1969random}).

Let $\cS_\aa(\sigma)$ be the symmetric $\alpha$-stable (\sas) law of index $\alpha\in (0,2]$ with $\sigma \geq 0$ being the scale 
parameter\footnote{Our approximation in Thm \ref{maineq2prop}, as well as the Lindeberg-Feller result, can  be extended to stable distributions 
with skewness $\nu\neq 0$, however,
to simplify calculations and notation we have assumed symmetry.}.
We denote the characteristic function of $\cS_\aa(\sigma)$ by \begin{equation}\label{eq:fcstable1}
\lambda_\alpha(\theta)=\exp\left(-|\sigma\theta|^\alpha \right), \quad \theta\in\bbR.
\end{equation}
To reduce notation, when $\sigma=1$ we simply write $\cS_\alpha=\cS_\aa(1)$.

We now consider the class of independently scattered \sas random measures, i.e. those where $M[f]$ is \sas for all $f\in V$. Suppose that $(E,\cE,m)$ 
is a measure space where  
$m$  is a $\sigma$-finite measure and  $\cE_0$ is the class of measurable sets with finite $m$-measure. Following \cite[Sec. 3.3]{samorodnitsky1994stable}, 
we say that the 
independently scattered \sas random measure $M_\alpha$ has {\it control measure} $m$
 if  $M_\alpha(A)$ has distribution $\cS_\alpha(m(A)^{1/\aa})$ for all $A\in\cE_0$.
For such random measures, it can be shown that $V=L^\aa(E)$ (see \cite[Ch. 3]{samorodnitsky1994stable}) and
that the distributions $M_\aa(A), A\in\cE_0$ uniquely determine the characteristic functions
\bes
\E\exp\{i\th M_\aa[f]\}=
\exp\lc -\int_E |\th f(x)|^\aa\, m(dx)\rc.
\ees
In the Gaussian case $\alpha=2$,  this is just the usual Wiener integral.

In the rest of this section we develop a discrete approximation of $M_\aa$ when $E=\bbR^d$ with Lebesgue control measure.
We begin by recalling that the domain of  attraction of $\cS_\aa $ consists of random variables $\xi$ such that
\begin{equation}\label{eq:xi1}
a_n^{-1}\Big(\sum_{k=1}^n \xi_k -b_n\Big) \Longrightarrow \cS_\aa  \quad \mathrm{as}\  n\to\infty,
\end{equation}
where $a_n>0$ and $b_n\in\bbR$ are normalization constants and the
$\xi_k$'s are i.i.d. copies of $\xi$. In the sequel, we will assume
$\eta$ is \sas, and $\xi$ is not only in the domain of attraction of
$\eta$, but also that the normalization constants are precisely
\begin{equation}\label{eq:xi2}
a_n=n^{1/\alpha}\quad \mbox{and}\quad b_n=0,\quad n\geq 1.
\end{equation}
When $\aa<2$, such distributions are said to be in the domain of {\it normal}
attraction of $\cS_\aa$ which is not to be confused with the normal
domain of attraction.

We propose a discrete approximation of $M_\alpha$ based on the lattice
$h\bbZ^d\subset \bbR^d$ with edge length $h$.  Let $(\xi_k)_{k\in\bbZ^d}$ be a random field of i.i.d. copies of $\xi$
satisfying \eqref{eq:xi1} and \eqref{eq:xi2} and formally define
\begin{equation}\label{eq:series}
 M_\aa^h[f]:=\sum_{k\in\bbZ^d} f^h(k)\xi_{k},
\end{equation}
where for $I^d=[0,1)^d$, $f^h:\Z^d\mapsto \R$ is
\begin{eqnarray}\label{def:superh}
f^h(k)&:=&\int_{h(k+I^d)} f(x) \,dx, \quad f\in L^1_{\text{loc}}(\R^d).
\end{eqnarray}
Note that we have implicitly fixed an enumeration $\{k_n, n\geq 1\}$
of $\Z^d$ and convergence of $\sum_{k\in\bbZ^d}a_k$ really means
convergence of $\sum_{n=1}^\infty a_{k_n}$. 

The discrete random measures $M^h_\aa$  approximate $M_\aa$ in the
following sense:
\begin{theorem}[Approximation of \sas random measures]\label{maineq2prop}
Fix $\aa\in(0,2]$.  If $\aa\in[1,2]$, let $f_t\in L^\aa(\R^d)$ for all $t$ in an
index set $T$.  If $\aa\in(0,1)$, for a fixed $\epsilon>0$  let $f_t\in L^{\aa-\epsilon}\cap L^1(\R^d)$ for all $t\in T$.
Then as $h\to 0$\beq\label{maineq2}
 M_\aa^h[f_t]\stackrel{fdd}{\longrightarrow} M_\alpha[f_t].
\ee
\end{theorem}
The notation $\stackrel{fdd}{\longrightarrow}$ denotes weak
convergence of the finite dimensional distributions, i.e.,
convergence in distribution of $ M_\aa^h[f]$   for all
linear combinations $f=\th_1f_{t_1}+\cdots +\th_nf_{t_n}$. When the
functions are indexed by one-dimensional time, it was shown in
\cite{avram1992weak}, that even for the simple family
$f_t=1_{[0,t]}\in L^\aa(\R)$, the above convergence does not hold in
the $J_1$-Skorokhod topology\footnote{ In \cite{avram1992weak}, it
was also shown that under the right conditions, convergence does
occur in Skorokhod's $M_1$ topology.}. Theorem \ref{maineq2prop}
will follow from a Lindeberg-Feller type result for stable
distributions which we state in Theorem \ref{theo:whitecase} below.

Let us make one more remark before stating Theorem \ref{theo:whitecase}. One motivation for \eqref{maineq2} was
to provide a means to simulate a  process
$X_t=\int_{\bbR^d}f_t(x)M_\alpha(dx)$. For such simulations, it is
natural to let the $\xi_k$'s be i.i.d. copies of $\cS_\aa $ (rather
than only in the domain of normal attraction). If one is concerned
only with one-dimensional distributions (a single function $f$),
then a better approximation is given by replacing
$f^h(k)$ in \eqref{eq:series} by \beq \label{eq:one-dd}
u_k:=\(\int_{h(k+I^d)}f(x)^{<\aa>} \,dx\)^{<1/\aa>} \ee
where we have
used the notation $x^{<\aa>}:=\mathrm{sign}(x)|x|^\alpha$. In fact,
using $u_k$, one can check that the approximation is exact, and the
right and left sides of \eqref{maineq2} are equal in distribution
for every $h>0$. The reason we have not used \eqref{eq:one-dd} for
the general approximation scheme is due to the fact that
 \eqref{maineq2} is no longer a \sas random measure under \eqref{eq:one-dd} because
 the linearity property \eqref{eq:lin} does not hold. The analysis of
the finite-dimensional distributions then becomes much more
difficult.

\begin{theorem}[Lindeberg-Feller type stable limit theorem]\label{theo:whitecase}
Suppose $(\xi_{k,j})_{k,j\in\N}$ is an i.i.d. array of random variables in the domain of normal attraction of $\cS_\alpha$, $\alpha\in(0,2]$, and $(u^{(j)})_{j\in\N}$ 
is a sequence of vectors in  $\ell^\alpha$, i.e.
$u^{(j)}:=(u^{(j)}_k)_{k\in\N}\in \ell^\alpha$ for all $j\in\N$.
If
\begin{enumerate}
\item $\lim_{j\to\ff} \|u^{(j)}\|_{\ell^\alpha}=\sigma$ and \\
\item $\lim_{j\to\ff} \|u^{(j)}\|_{\ell^\infty}=0$
\end{enumerate}
then
$
\sum_{k}u_k^{(j)}\xi_{k,j}<\ff
$
a.s. for each $j\in\N$ and
$$
\sum_{k\in\N} u_k^{(j)}\xi_{k,j} \Longrightarrow \cS_\alpha(\sigma)  \quad \mathrm{as}\  j\to\infty.
$$
\end{theorem}

\noindent
{\bf Remarks:}
\begin{enumerate}
\item The condition that the $\xi_{k,j}$ be identically distributed can be relaxed slightly to the condition that
$\E[ \exp(i\theta \xi_{k,j}) ] = 1-|\theta|^\alpha + o(|\theta|^\alpha)$
holds uniformly in $k,j$ as $\theta\to 0$.  For example,
they may be chosen from a finite family of distributions in the domain of normal attraction
of $\cS_\aa$.
\item
The a.s. convergence
$\sum_{k\in\N} u_k^{(j)}\xi_{k,j}<\infty$ 
in fact occurs if and only if $u=(u_k)_{k\in\N}\in\ell^{\aa}$ as will be seen in Lemma \ref{lem1}.
\item Although the series $\sum_{k\in\N} u_k^{(j)}\xi_{k,j}$ may not converge absolutely, switching the order of summation does not change the
convergence in distribution to $\cS_\aa(\sigma)$.  This will be apparent in the proof.
\item In the Gaussian case, the result can be seen as a variant of the usual Lindeberg-Feller Theorem by noticing
that condition 2, concerning $\ell_\infty$, is equivalent to
$$\lim_{j\to\ff}\sum_k 1{\{|u^{(j)}_k|>\eps\}}=0$$
for all $\eps>0$.
More generally when \mbox{$0<\aa\le 2$},  the result is related to
Theorem 3.3 of \cite{petrov1995limit} which gives necessary and
sufficient conditions for convergence of sums of independent
triangular arrays to a given infinitely divisible distribution. In
particular, the conditions of Theorem \ref{theo:whitecase} above
imply the {\it infinite smallness} condition (cf. Eq. (3.2) in
\cite{petrov1995limit}). However, it is unclear how to obtain
Theorem \ref{theo:whitecase} from \cite[Thm 3.3]{petrov1995limit} in
a manner simpler than the proof of Theorem \ref{theo:whitecase}
provided below.
\end{enumerate}

\section{Linear fractional stable random measures}\label{sec:lfsm}

To simplify matters, in this section we will restrict our attention to the one-dimensional case $E=\R^1$ equipped with Lebesgue measure.
For higher dimensions, see the first remark following Corollary \ref{theo:fraccase}.
Also, in this section we assume that $1<\aa\le 2$.

\subsection{Fractional integro-differentiation and L-FSM integrals}
In this subsection we define the stochastic integration of suitable functions with respect to different L-FSMs in terms of
 stable random measures which are not independently scattered.  This is achieved using fractional integrals and derivatives.  The intuition behind our definition
is based on two facts. The first is that fractional integrals and derivatives can be realized using convolutions, and the second is that convolutions are moving averages.

The practice of using fractional integro-differentiation for analogous integrals with respect to FBM was initiated in \cite{decreusefond1999stochastic}, 
and was subsequently used in \cite{pipiras2000integration}. We note that the $M$ operator, which is fundamental in the development 
of the so-called WIS integral (\cite{elliot2003}), 
is simply  fractional integro-differentiation in disguise.

Before we define our integral, let us review some preliminaries concerning fractional integro-differentiation.
The Riemann-Liouville integrals are defined, for  $ f \in L^p(\R), 1\le p<1/ \delta$ and $0< \delta<1$, by
\bea\label{fracintegral}
(I^ \delta_{+} f )(x)&:=& \frac{1}{\Gamma( \delta)} \int_{-\ff}^x \frac{ f (t)}{(x-t)^{1- \delta}} \,dt \\
&=& \frac{1}{\Gamma( \delta)} \int_\R \frac{ f (t)}{(x-t)_+^{1- \delta}} \,dt\nn\\
(I^ \delta_{-} f )(x)&:=& \frac{1}{\Gamma( \delta)} \int_{\R} \frac{ f (t)}{(x-t)_-^{1- \delta}} \,dt
\eea
Our notation is consistent with the standard reference on this topic, \cite[Sec. 5.1]{samko1987integrals}, where some basic properties of the above can be found.
For example, if $ f $ is in the  Schwartz space and we allow for $ \delta\in\N$,
then \eqref{fracintegral} gives the usual integral, as can be seen by Cauchy's formula for repeated integration:
\bea
\int_{-\ff}^x\int_{-\ff}^{t_{n}}\cdots\int_{-\ff}^{t_2}  f (t_1)\, dt_1 \cdots dt_{n-1}dt_n = \frac{1}{(n-1)!}\int_{-\ff}^x(x-t)^{n-1} f (t) \,dt.\nn
\eea
Also, the above fractional integrals have the semigroup property for  $ \delta,\gamma>0$ and $ \delta+\gamma<1$:
\beq\nn
I^ \delta_{\pm}I^\gamma_{\pm} f  = I^{ \delta+\gamma}_{\pm} f .
\ee
For sufficiently nice $ f $, this semigroup property extends to all $ \delta,\gamma>0$.

Suppose $f\in \CC^1 $ and  $f'\in L^{1}$. These are sufficient conditions for the following
 Riemann-Liouville derivatives to exist:
\bea\label{fracderiv}
(\DD^\bb_{+} f )(x)&:=& \frac{1}{\Gamma(1-\bb)} \frac{d}{dx}\int_\R \frac{ f (t)}{(x-t)_+^{\bb}} \,dt\\
(\DD^\bb_{-} f )(x)&:=& \frac{1}{\Gamma(1-\bb)}\frac{d}{dx} \int_{\R} \frac{ f (t)}{(x-t)_-^{\bb}} \,dt. \nn
\eea
If $ f \in L^1$, it is known that the inversion $\DD_{\pm}^\bb I^\bb_\pm  f = f $ holds.

Bringing the derivative inside the integral in \eqref{fracderiv},
the Riemann-Liouville integrals and derivatives of $ f $ can be seen as convolutions of $f$ and $ f'$ with the family
\beq\label{def:w}
w_{a,b}(x)=w_{a,b}^{(\bb)}(x):=ax_-^{-\bb} + bx_+^{-\bb}, \quad \bb\in(0,1)
\ee
where we have set $\beta=1-\delta$.

\begin{definition}[Linear fractional stable random measures]\label{maindef}
Fix $1<\aa\le 2$ and $a,b\ge 0$.
\begin{enumerate}
\item

If $\bb\in(1/\aa,1)$, let $f\in L^1\cap L^\aa $.

The {\it linear fractional random measure} with {\it long range dependence} is defined by
\bea\label{posintegral}
 M_{\alpha,H} [f]&:=&
M_\aa[aI^\bb_- f+b{I^\bb_+ f}]=M_{\alpha}[f\ast w_{a,b}^{(\bb)}]
\eea
where the Hurst parameter is given by $H=1+1/\aa-\bb$.
\item
If $\bb\in(0,1/\aa)$, let $f\in \CC^1 $ and $f'\in L^1\cap L^\alpha $.

The {\it linear fractional random measure} with {\it anti-persistence} is defined by

\bea\label{negintegral}
 M_{\alpha,H} [f]&:=&
M_\aa[a\DD^\bb_- f+b{\DD^\bb_+ f}]=M_{\alpha}[\frac{d}{dx}(f\ast w_{a,b}^{(\bb)})]
\eea
where $H=1/\aa-\bb$.
\end{enumerate}
\end{definition}

It is not hard to check that $I^\bb_\pm f$ and $\DD^\bb_\pm f$ are in $L^\aa$ so that \eqref{posintegral} and \eqref{negintegral} are well-defined: to see this, 
split $w_{a,b}^{(\bb)}$ into an $L^1$ and $L^\aa$ function using
$1_{[-\epsilon,\epsilon]} +1_{[-\epsilon,\epsilon]^c}$ and apply Young's convolution inequality,
\beq\label{Young}
\| f\ast g\|_r \le \|f\|_p \|g\|_q \quad\text{ for }\, \frac{1}{p}+\frac{1}{q}=\frac{1}{r}+1.
\ee
In fact, one can slightly improve the condition for \eqref{posintegral} to $f\in L^{\alpha(1+\alpha(1-\bb)+\epsilon)^{-1}}\cap L^{\alpha(1+\alpha(1-\bb)-\epsilon)^{-1}}$
for some $\eps>0$,
and a similar condition can be found for \eqref{negintegral} and $f'$.  However, in the interest of simple notation, we will not utilize these meager improvements in the sequel.
Let us remark that the fact that \eqref{posintegral} is well-defined coincides with Proposition 3.2 in \cite{pipiras2000integration} for the Gaussian case.

By the linearity of convolutions, it follows that $M_{\alpha,H}$ is a \sas random measure.
Also note that $M_{\alpha,H} [f]$ can be interpreted as the integral of $f$ with respect to a L-FSM in which case we write
\beq
M_{\alpha,H} [f]\equiv \int_{\R} f\, dL_{\aa,H}.
\ee

To check consistency with \eqref{LFSMkernel}, we see that
\bea
&&M_{\alpha}[1_{[0,t]}\ast w_{a,b}]\\
\nn&=&\int_{\bbR}\left(\int_{\bbR}1_{[0,t]}(y)\(a(x-y)_-^{-\bb}+b(x-y)_+^{-\bb}\)dy\right)M_\alpha(dx)\\
&=&\nn\int_{\bbR}\left(\int_{\bbR}1_{[0,t]}(y)\(a(y-x)_+^{-\bb}+b(y-x)_-^{-\bb}\)dy\right)M_\alpha(dx)\\
&=& \nn\int_{\bbR} f_t^{a,b}(x) \,M_\alpha(dx)
\eea
and
\bea
&&M_{\alpha}[\frac{d}{dx}(1_{[0,t]}\ast w_{a,b})(x)]\\
&=&\int_{\bbR}\frac{d}{dx}\int_{\bbR}1_{[0,t]}(y)\(a(x-y)_-^{-\bb}+b(x-y)_+^{-\bb}\)dy\,M_\alpha(dx) \nn\\
&=& \nn\int_{\bbR} f_t^{a,b}(x) \,M_\alpha(dx).
\eea

When $f\in \CC^1 $ and that $f'\in L^{1}$, one can rewrite \eqref{fracderiv} as
\bea
&&\frac{1}{\Gamma(1-\bb)} \int_\R \frac{ f' (x-t)}{t_+^{\bb}} \,dt\nn\\
&=&\frac{\bb}{\Gamma(1-\bb)} \int_0^\ff f' (x-t) \int_t^\ff \frac{1}{s^{1+\bb}} \,ds\\
&=&\frac{\bb}{\Gamma(1-\bb)} \int_0^\ff \frac{f (x) - f(x-s)}{s^{1+\bb}} \,ds.\nn
\eea
The right-hand side above is slightly more general then \eqref{fracderiv} and is called the  {\it Marchaud derivative}. This is the fractional derivative used
 in \cite{pipiras2000integration}, however, to keep a unified notation in our approximations of the next subsection, we will continue with the Riemann-Liouville derivative.

\subsection{Discrete approximations of linear fractional stable measures}

Let $1<\alpha\leq 2$, and consider the stationary moving average process $(\hat\xi_k)_{k\in\bbZ}$ obtained by ``linearly filtering" an  i.i.d. sequence $(\xi_l)_{l\in\bbZ}$
in the domain of normal attraction of  $\cS_\aa$:
\begin{equation}\label{eq:defhatxi}
\hat\xi_k:=\sum_{l\in\bbZ}\B_{k-l}\xi_l.
\end{equation}
Lemma \ref{lem1} shows that if
$\B\in\ell^\alpha$, the series \eqref{eq:defhatxi} converges almost
surely.

Recall the definition of $f^h$ from \eqref{def:superh} and denote the inversion of a sequence by $\check\B_k:=\B_{-k}$
A first stab at approximating a L-FSM integral of $f$, as defined in the previous subsection, might be to mimic
\eqref{eq:series} and look at  $\sum_{k\in\bbZ} f^h_k\hat\xi_{k}$ for appropriate filters $v$ (which would also depend on $h$). 
This is, for example, the approach of \cite{kasahara1988weighted}
and \cite{maejima2008limit}.
Then formally,
\begin{eqnarray}
\label{convolution convergence} \sum_{k\in\bbZ}
f^h_k\hat\xi_{k}&=&\sum_{k,l\in\bbZ}
f^h_k\B_{k-l}\xi_l \\&=&\sum_{l\in\bbZ}
\(f^h\ast\check\B\)_l\xi_l <\ff.\nn
\end{eqnarray}

However, in view of the right-hand side above, it is easier and perhaps more natural to first convolve $f$ with $w_{a,b}=w_{a,b}^{(\beta)}$ and then 
approximate the convolution on a lattice with side-length $h$.  In particular, for $w_{a,b}$ corresponding to $H\in(1/\aa,1)$, define
\begin{equation}\label{eq:series2}
M^h_{\alpha,H}[f]:=\sum_{k\in\bbZ} \(f\ast w_{a,b}\)^h_k \xi_{k},\quad \quad f\in L^1\cap L^{\aa}(\R)
\end{equation}
where the sequence $\(f\ast w_{a,b}\)^h$ is defined according to \eqref{def:superh}. Alternatively, for $w_{a,b}$ corresponding to $H\in(0,1/\aa)$, define
\begin{equation}\label{eq:series2n}
M^h_{\alpha,H}[f]:=\sum_{k\in\bbZ} \(f'\ast w_{a,b}\)^h_k\xi_{k},\quad \quad f\in \CC^1(\R), f'\in L^1\cap L^{\aa}(\R).
\end{equation}
By \eqref{Young} and the remark above it,  $f\ast w_{a,b}\in L^\aa(\R)$.  Thus one obtains, from a direct application of Theorem \ref{maineq2prop}, the following corollary:

\begin{corollary}\label{theo:fraccase}
Fix  $1<\alpha\leq 2$. Suppose that for all $t$ in an index set $T$, $f_t\in L^1\cap L^\aa$  when $H\in(1/\alpha,1)$ 
 or  $f_t\in \CC^1, f'_t\in L^1\cap L^\aa$ when $H\in(0, 1/\alpha)$. Then as $h\to 0$:
\beq\label{theo:fraccase equation}
M^h_{\alpha,H}[f_t]\stackrel{fdd}\longrightarrow  M_{\alpha,H} [f_t].
\ee
\end{corollary}

\noindent
{\bf Remarks:}
\begin{enumerate}
\item It is not hard to extend the $H>1/\aa$ case to stable random measures on $\R^d$ by generalizing the two fixed values $a,b\ge 0$
 (representing the negative and positive directions) to a function on the unit sphere $S^{d-1}\subset \R^d$. 
  However, one then has to specify what is meant by ``stationary increments'' as there are different possibilities for $d>1$.
\item Extending the $H<1/\aa$ case to higher dimensions is more difficult. One possibility is to consider
 the Marchaud derivative in place of the Riemann-Liouville derivative (see also the next remark).
\item When $H>1/\aa$, Eq. \eqref{theo:fraccase equation} has been shown by various authors in the case where $f_t=1_{[0,t]}$ (see \cite{kasahara1988weighted} and its references).
 However,
when $H<1/\aa$, to our knowledge, even the case $f_t=1_{[0,t]}$ has not appeared in the literature.  It is, however, related to the normalization suggested in Theorem 5.2 of \cite{kasahara1988weighted} which can be thought of as a discrete Marchaud derivative in the case where $f_t=1_{[0,t]}$.
\end{enumerate}

\section{Proofs}\label{sec:proof}
Before delving into the proofs, let us recall some facts about the
domain of  attraction of a stable distribution.  
We write
$f(x)\sim g(x)$ as $x\to c$ if $\lim_{x\to c}f(x)/g(x)=1$.
For $\alpha\in(0,2)$, the following statements
are  equivalent (see \cite[Theorem 1]{geluk1997stable} with
$p=1/2$):
\begin{itemize}
\item[i)] $\xi$ is in the domain of attraction of $\cS_\alpha$ (i.e. Eq. \eqref{eq:xi1} holds);
\item[ii)] the tail function $t\mapsto \bbP(|\xi|\geq t)$ is regularly varying at infinity with index $-\alpha$ and $\bbP(\xi\leq -t)\sim \bbP(\xi\geq t)$ as $t\to \infty$;
\item[iii)] the characteristic function $\lambda(\theta)=\bbE\left[e^{i\theta\xi}\right]$ satisfies
\begin{itemize}
\item[-] $\theta\mapsto 1- \mathrm{Re}(\lambda(\theta))$ is regularly varying at $0$ with index $\alpha$,
\item[-] for all $x\neq 0$,
\[
\lim_{\theta\to 0} \frac{x\mathrm{Im}(\lambda(\theta x))-\mathrm{Im}(\lambda(\theta ))}{1- \mathrm{Re}(\lambda(\theta))}=0.
\]
\end{itemize}
\end{itemize}
Moreover, if conditions (i)-(iii) hold, then
\[
1- \mathrm{Re}(\lambda(\theta))\sim c_\alpha \bbP(|\xi|\geq 1/\theta)\quad \mbox{with}\ c_\alpha=\int_0^\infty x^{-\alpha}\sin x\, dx
\]
as $\theta\to 0$ and also
\[
\mathrm{Im}(\lambda(\theta))=\theta \int_0^{1/\theta}(\bbP(\xi\geq s)-\bbP(\xi\leq-s))ds+ o(\bbP(|\xi|\geq 1/\theta)).
\]
Also, Remark 3 of \cite{geluk1997stable} shows that one may choose the normalization constants  so that
\[
\lim_{n\to \infty} n\Big(1- \mathrm{Re}(\lambda(1/a_n))\Big) =1\quad \mbox{and}\quad b_n=n\mathrm{Im}(\lambda(1/a_n)).
\]

Recall from \eqref{eq:xi2} that in the present framework, we have assumed  $a_n=n^{1/\alpha}$ and $b_n=0$.
Thus,
\begin{equation}\label{eq:heavy-tail}
\bbP( \xi\geq t) \sim \bbP( \xi\leq -t)\sim \frac{1}{2c_\alpha} t^{-\alpha} \quad \mbox{as}\ \theta\to 0.
\end{equation}
and
\begin{equation}\label{eq:fcstable2}
\lambda(\theta)=1-|\theta|^\alpha+o(|\theta|^\alpha)=\lambda_\alpha(\theta)+o(|\theta|^\alpha)\quad \mbox{as}\ \theta\to 0.
\end{equation}
where $\lambda_\alpha$ is defined in Eq. \eqref{eq:fcstable1}. Furthermore, (\ref{eq:heavy-tail}) implies that there exists $C>0$ such that for any $s>0$
\begin{equation}\label{eq:stablestim2}
{\rm Var}[\xi\ind_{\{|\xi|\leq s\}}]\leq Cs^{2-\alpha}\quad {\rm and} \quad  \E[|\xi|\ind_{\{|\xi|\leq s\}}]\leq Cs^{1-\alpha}.
\end{equation}

\subsection{Proof of Theorem \ref{theo:whitecase}}

We begin with a lemma which shows the $\ell^\alpha$ is the right space for the sequence $u$.
\begin{lemma}\label{lem1}
If $(\xi_k)_{k\in\N}$ is an i.i.d. sequence in the domain of normal attraction of $\cS_\aa $, then
$\sum_{k}u_k\xi_{k}<\ff$ if and only if $u\in\ell^\alpha$.
\end{lemma}
\begin{proof}[Proof of Lemma \ref{lem1}]
The case $\alpha=2$ is standard and omitted. Consider  $\alpha\in (0,2)$.
Recall Kolmogorov's Three-series Theorem: $\sum_{k} u_k\xi_k$  converges a.s.  if and only if for any $s>0$, the following
three series converge
\[
\sum_{k\in \N} \bbP\left[|u_k\xi_k |>s\right],\quad
\sum_{k\in \N} {\rm Var}\left [u_k\xi_k \ind_{\{|u_k\xi_k |\leq s\}}\right],
\quad
\sum_{k\in \N} \E\left[u_k\xi_k\ind_{\{|u_k\xi_k |\leq s\}}\right] .
\]
Eq. (\ref{eq:heavy-tail}) implies
\begin{equation}\nn
\bbP\left[|u_k\xi_k |>s\right]\sim C|u_k|^\alpha s^{-\alpha}
\end{equation}
and hence the first series converges if and only if $u\in \ell^\alpha$.
If $u\in\ell^\alpha$, then (\ref{eq:stablestim2}) implies the convergence of the third series since
$$|u_k| \E[|\xi_k| 1_{\{|\xi_k|<s/|u_k|\}}]  \leq  |u_k| C(s/|u_k|)^{\alpha-1} = Cs^{\alpha-1} |u_k|^\alpha.$$
The convergence of the second series is obvious.
\end{proof}

\begin{proof}[Proof of Theorem \ref{theo:whitecase}]
If $\eta_{k,j}$ are i.i.d. $\cS_\aa$ random variables, then for any fixed $j$ the above lemma allows us to write
\beq\label{eq:fc-stableintegral}
\E\exp\{i\th \sum_{k\in\N} u^{(j)}_k\eta_{k,j}\}=\prod_{k\in\N} \lambda_\aa\left( u^{(j)}_k\th\right)=
\exp\lc -\left|\|u^{(j)}\|_{\ell_\aa}\th\right|^\aa\rc
\ee
and
\begin{equation}\label{eq:fc-noise}
\E\exp\{i\theta \sum_{k\in\N} u^{(j)}_k\xi_{k,j}\}=\prod_{k\in\N} \lambda\left(   u^{(j)}_k\th \right).
\end{equation}
Note that since $\|u\|_{\ell_\aa}^\aa:=\sum_{k\in\N} |u(k)|^\aa$ absolutely converges, the order in which the summation and  products above are taken is irrelevant.

It suffices to show that as $ j\to\ff$,
\begin{equation}\label{eq:diff}
\prod_{k\in\N} \lambda\left(   u^{(j)}_k\th \right)=\prod_{k\in\N} \lambda_\aa\left(   u^{(j)}_k\th \right)+o(1).
\end{equation}
We fix $j$ and estimate the difference of the above products using the following fact : if $(z_i)_{i\in I}$ and $(z_i')_{i\in I}$ are two families of complex numbers with moduli 
no greater than $1$ and such that the products $\prod_{i\in I}z_i$ and $\prod_{i\in I}z'_i$ converge, then
\beq\label{standardinequality}
\left|\prod_{i\in I}z'_i -\prod_{i\in I}z_i\right| \leq \sum_{i\in I} \left|z'_i - z_i\right|.
\ee
We therefore have
\begin{eqnarray}
\left|\prod_{k\in\N} \lambda\left(   u^{(j)}_k\th \right)- \prod_{k\in\N} \lambda_\aa\left(  u^{(j)}_k\th \right)\ \right|
 \leq \  \sum_{k\in\N}  \ \left|\ \lambda \left(u^{(j)}_k\th \right)-\lambda_\alpha \left(u^{(j)}_k\th \right)\  \right|\label{eq:diff1}.
\end{eqnarray}
Equation (\ref{eq:fcstable2}) implies\footnote{We have assume $\alpha\in(0,2)$ for Eq. \eqref{eq:fcstable2}, but for $\alpha=2$ it is well-known.} that the 
function $g$ defined by $g(0)=0$ and
$$
g(u)=|u|^{-\alpha}\left|\lambda(u)-\lambda_\aa(u)\right|\ \ ,\ \ u\neq 0,
$$
is continuous and bounded and for any $k\in\N$, we have
$$
\left| \lambda \left(  u^{(j)}_k\th  \right)-  \lambda_\aa \left(  u^{(j)}_k\th  \right) \right| = g( u^{(j)}_k\th  )| u^{(j)}_k\th  |^\alpha .
$$
In order to obtain a uniform estimate on the above, define the function $\tilde g:\R^+\to\R^+$ by
$$
\tilde g(v) :=\sup_{|u|\leq v} |g(u)|.
$$
Note that $\tilde g$ is continuous, bounded and vanishes at $0$, and that for any $k\in \N$ such that $| u^{(j)}_k\th  |\leq \varepsilon$,
\begin{equation}\label{eq:diff2}
\left| \lambda \left(  u^{(j)}_k\th  \right)-  \lambda_\aa \left(  u^{(j)}_k\th  \right) \right|\leq \tilde g(\varepsilon)| u^{(j)}_k\th  |^\alpha .
\end{equation}
Let $\varepsilon>0$. Equations (\ref{eq:diff1}) and (\ref{eq:diff2}) together yield
\begin{eqnarray*}
& &\left|\prod_{k\in\N} \lambda\left(   u^{(j)}_k\th \right)- \prod_{k\in\N} \lambda_\aa\left(  u^{(j)}_k\th \right)\ \right|\\
&\leq& \  \tilde g(\varepsilon)\sum_{k\in\N}\left| u^{(j)}_k\th   \right|^\alpha\ind_{\{| u^{(j)}_k\th  |\leq\varepsilon\}} + 2\sum_{k\in\N}\ind_{\{| u^{(j)}_k\th  |>\varepsilon\}}.
\end{eqnarray*}
Now, by the continuity of $\tilde g$ at  $0$,  $\tilde g(\varepsilon)$ is small when $\varepsilon$ is small. Eq. \eqref{eq:diff} follows since  
$\lim_{j\to\ff} \|u^{(j)}\|_{\ell^\infty}=0$ implies
$\sum_{k\in\N}\ind_{\{| u^{(j)}_k\th  |>\varepsilon\}}\to 0$ as $j\to\ff$.
\end{proof}

\subsection{Proof of Theorem \ref{maineq2prop}}
Let $\lfloor\cdot\rfloor$ denote the floor function applied to each coordinate of $\R^d$.
Define $\fh:\R^d\mapsto\R$ to be a
piece-wise constant function approximating $f\in L^1_{\text{loc}}(\R^d)$:
\bea\label{def:fh}
\fh (x)&:=&  \int_{h(\lfloor h^{-1}x\rfloor+I^d)}h^{-d} f(y) \, dy\\
&=&\int_{h(k+I^d)} h^{-d} f(y) \,dy, \ \, \quad\text{for } x\in h(k+I^d)\nn\\
&=& h^{-d} f^h(k), \quad\quad\quad\quad\quad\quad\text{for } x\in h(k+I^d).\nn
\eea
Note that \beq\label{normsequal}
\|f^h\|_{\ell^\alpha}=\|\fh\|_{L^\alpha}.
\ee

\begin{lemma} \label{laa lemma}For $\aa\in[1,2]$, suppose $f\in L^\aa(\R^d)$. Then as $h\to 0$,
$$\lim_{h\to 0}\| \fh-f\|_{L^\aa}=0.$$
\end{lemma}
\begin{proof}[Proof of Lemma \ref{laa lemma}]
To reduce notation we assume $d=1$, but the proof holds for general $d$. Fix $k\in\Z$ and consider the sequence of $h$'s such that  $h=2^{-j}$ for  $j\in\N$.
We will exploit the fact that  
$\fh1_{[k,k+1)}$ is a martingale (in time $j$)  with respect to Lebesgue measure on $[k,k+1)$ and with respect to the $\sigma$-fields
generated by the sets  $2^{-j}[i,i+1), i\in\Z$. 

For $\aa\ge 1$, $|\fh|^\aa1_{[k,k+1)}$ is a submartingale which, by the martingale convergence theorem, converges a.s. to $|f|^\aa1_{[k,k+1)}$. 
Thus, Fatou's lemma gives
\beq\label{eq:aag1}
\|\fh1_{[k,k+1)}\|_{L^\aa}^\aa\to \|f1_{[k,k+1)}\|_{L^\aa}^\aa.
\ee
Since $\fh1_{[k,k+1)}$ converges a.s. and the $L^\aa$-norms converge, we have convergence in $L^\aa(\R)$ of $\fh1_{[k,k+1)}$ and also for $\fh1_{[-N,N)}$ for any $N\in\N$.

For $f\in L^\aa(\R)$ without compact support,
simply choose $N$ so that $$\|f1_{[-N,N)^c}\|_{L^\aa}^\aa<\epsilon.$$
Since $|\fh|^\aa1_{[k,k+1)}$ is a submartingale, we also have \mbox{$\|\fh1_{[-N,N)^c}\|_{L^\aa}^\aa<\epsilon$} uniformly in $h$.

Finally, to extend the above to general $h\to 0$.  Note that all we really require is a sequence of lattices such that finer lattices are sublattices of 
prior ones and that the mesh size goes to zero.  But any such sequence has  the same limit in $L^\aa(\R)$, thus we conclude that the only real 
requirement is that the mesh size goes to zero.
\end{proof}
\begin{proof}[{Proof of Theorem \ref{maineq2prop}}]
By the Cr\'amer-Wold device, we must show that for all
$\theta_1,\ldots,\theta_n\in\bbR$ and $f_1,\ldots,f_n\in
L^\aa(\R^d)$,
\[
\sum_{i=1}^n \theta_iM_\alpha^h[f_i]\Longrightarrow \sum_{i=1}^n \theta_iM_\alpha[f_i]\quad \mbox{as}\ h\to 0.
\]
Our proof uses Theorem \ref{theo:whitecase}. First note that the comment following \eqref{eq:fc-noise} shows that switching the order of 
summation in the series $M_\alpha^h[f_i]$ does not affect its distribution.  This, together
with the linearity of $M_\alpha$ and $M_\alpha^h$,  allows us to reduce the above
to verifying
\begin{equation}\nn
M_\alpha^h[f]\Longrightarrow M_\alpha[f]\quad \mbox{as}\ h\to 0
\end{equation}
for a single $f\in L^\aa(\R^d)$. This will follow from Theorem \ref{theo:whitecase} provided we check the two conditions
\begin{equation}\label{eq:cond1}
\lim_{h\to 0}\|f^h\|_{\ell^\alpha}=\|f\|_{L^\alpha}
\end{equation}
and
\beq\label{eq:cond2}
\lim_{h\to 0}\|f^h\|_{\ell^\infty}=0.
\ee

We consider $\aa\in[1,2]$ first. Condition \eqref{eq:cond1} easily follows from \eqref{normsequal} and Lemma \ref{laa lemma}. For \eqref{eq:cond2},
note that convergence of the $L^1$ norms of $|\fh|^\aa$, coupled with a.e. convergence, shows that the family $\{|\fh|^\aa\}_{h\in\N}$ is uniformly integrable.

For $\aa\in(0,1)$, we first consider the sequence of $h$'s such that  $h=2^{-j}$ for  $j\in\N$. By
 uniform integrability and the martingale convergence theorem (see the proof of Lemma \ref{laa lemma}), we see that $\fh1_{[k,k+1)}$ converges in $L^1$ to $f1_{[k,k+1)}$.
The final comment in the proof of Lemma \ref{laa lemma} shows the convergence also holds for arbitrary $h\to 0$. 
 
  Next, note that \mbox{$L^1([k,k+1))$} contains $L^\aa([k,k+1))$
and that the endomorphism on $L^1([k,k+1))$ which maps $$f1_{[k,k+1)}\ \mapsto \ |f|^\aa1_{[k,k+1)}$$ is continuous. Thus Eq. \eqref{eq:aag1} holds for $\aa\in(0,1)$.  

Since $f\in L^\aa$ we can  choose $N_1$ so that $\|f1_{[-N_1,N_1)^c}\|_{L^\aa}^\aa$ is small.
However, to uniformly bound the tails of the $f_h$, we will use the stronger condition of $f\in L^{\aa-\epsilon}$. 
In particular, there exist $N_2>0$, $C>0$ and $\delta>\alpha^{-1}$ such that 
$|x|\ge N_2$ implies $|f(x)|\leq C|x|^{-\delta}$. We have for $|x|\geq N_2+{h}$ that $${h}(\lfloor{h}^{-1}x\rfloor+I)\subset (-N_2,N_2)^c$$ and 
\begin{equation}\label{eq:maj-unif}
|\fh(x)|^\alpha=\left|{h}^{-1}\int_{{h}(\lfloor{h}^{-1}x\rfloor+I)}f(y) \, dy \right|^\alpha\leq C^\aa{h}^{1-\alpha}(|x|-{h})^{-\alpha\delta}.
\end{equation}
Since $\alpha\delta>1$, \eqref{eq:cond1} follows from \eqref{eq:aag1}.
Finally, as before,  we see that \eqref{eq:cond1} along with a.e. convergence gives \eqref{eq:cond2}
for $\aa\in(0,1)$.
\end{proof}

\section*{Acknowledgements}
We are grateful to Gennady Samorodnitsky for helpful correspondence.
\appendix

\end{document}